\documentclass[11pt]{article} 

\usepackage[utf8]{inputenc} 

\usepackage{geometry} 
\usepackage{graphicx} 
\usepackage{amsmath} 
\usepackage{amsfonts} 
\usepackage{amsthm} 
\usepackage{amssymb} 
\usepackage{enumitem}

\usepackage{algorithm}
\usepackage[noend]{algpseudocode}
\usepackage{authblk}
\usepackage{eucal}
\usepackage{url}
\usepackage{csvsimple}
\usepackage{sectsty}
\usepackage{verbatim}
\usepackage{longtable}

\sectionfont{\fontsize{12}{15}\selectfont}

\newtheorem{thm}{Theorem}[section]
\newtheorem{lem}[thm]{Lemma}

\newtheorem{cor}[thm]{Corollary}

\newtheorem{prop}[thm]{Proposition}

\title{On the Asymptotic Order of Circuit Codes}
\author{Kevin M. Byrnes and Florin Spinu
\thanks{E-mail:\texttt{dr.kevin.byrnes@gmail.com, fspinu@gmail.com}}
}

\begin{document}
\maketitle

\begin{abstract}
In this note we prove that the maximum length
of a $d$-dimensional circuit code of spread $k$ equals
$2^{d+O_k(\log^2d)}$, with the implied constant depending only on $k$.
\end{abstract}

\section{Introduction}
Let $I(d)$ denote the graph of the $d$-dimensional hypercube. 
For $k\geq 1$ an integer, a circuit code of \emph{spread} $k$, 
or a  $(d,k)$ circuit code, 
is a simple cycle $C$ in $I(d)$ satisfying the distance property
 $d_{I(d)}(x,y) \ge \min \{d_C(x,y),k\},\forall x,y\in C$.
Let $K(d,k)$ denote the maximum length of such a cycle.
In this paper we study its asymptotic order
\begin{equation}
\label{nu_def}
\nu(k):=\lim_{d\to\infty}\frac{\log_2 K(d,k)}{d}\ .
\end{equation}
For small $k$, existing results \cite{AK}:
$K(d,1)=2^d$ and $K(d,2)\ge \frac{3}{10}2^d$ imply $\nu(1)=\nu(2)=1$. 
For $k\ge 3$, it is not a priori clear that the limit exists or whether it is bounded away from zero. 
General estimates (\cite{Grunbaum} Chapter $17$) imply
the weaker bound $\nu(k) \ge 2/k$ or $2/(k+1)$ (depending on $k$ being even or odd); 
on the other hand, special constructions in certain dimensions $d=2^n+O(n)$ 
lead to  especially long codes.  
In this note we prove 
\begin{thm}
\label{thm1}
For any $k\ge 1$, 
$\log_2 K(d,k)=d+O(\log^2d)$ as $d\to\infty$, 
with the implied constant depending on $k$ only.  
In particular, $\nu(k)=1$.
\end{thm}

\section{Previous Results and Technical Lemmas}
Our notation follows \cite{Byrnes2019}, throughout this note we treat vertices of $I(d)$ and their corresponding binary vectors of length $d$ interchangeably.
A $(d,k)$ circuit code $C=(x_1,\ldots,x_N)$ starting at $x_1=\vec{0}$ is uniquely determined by its
\emph{transition sequence} $T(C)=(\tau_1,\ldots,\tau_N)$, 
where $\tau_i$ is the single position in which $x_i$ and $x_{i+1}$  (where $x_{N+1}=x_1$) differ.
Let $\omega$ denote a \emph{segment} (cyclically consecutive subsequence) of $T(C)$, and let $K(d,k,s)$ denote
the maximum length of a $(d,k)$ circuit code whose length is also divisible by $s$.

\begin{lem}
\label{lem_lb}
Let $k\ge 2$ and $d\ge (k+2)^2$, then $K(d,k)>2(d-1)k$.
\end{lem}
\begin{proof}
If $k=2$ the claim follows as $K(d,2)\ge \frac{3}{10}2^d$.
For $k\ge 3$, as $K(d,k)\ge K(d,k+1)$, the following lower bound (due to \cite{Singleton}) 
\begin{equation}
\label{singleton_lb_odd}
K(d,k)\ge (k+1) 2^{\lfloor 2d/(k+1) \rfloor -1}, \text{when $k$ odd and } \lfloor 2d/(k+1) \rfloor \ge  2
\end{equation}
can be applied to either $K(d,k)$ or $K(d,k+1)$, yielding the result.
\end{proof}

\begin{lem}
\label{lem_divis}
Let $k\ge 2$, $d\ge (k+2)^2$, and $s\in \{k,\ldots,d\}$.  Then $K(d+1,k,s)\ge K(d,k)$.
\end{lem}
\begin{proof}
Let $C$ be a $(d,k)$ circuit code of maximal length $N=K(d,k)$.
If $N/2$ is divisible by $s$, we embed $C$ trivially in dimension $d+1$, proving the claim.
Otherwise, $N/2=qs+r$ where $q=\lfloor N/2s \rfloor$ and $1\le r\le s-1$.  Let $p=s-r$ and partition $T(C)$ into $T(C) = (\omega_1,\omega_2)$ where each $\omega_i$ is a segment of length $N/2$.
For $i=1,2$ we form $\omega_i'$ by appending the new transition element $d+1$ (corresponding to going from dimension $d$ to dimension $d+1$) to the end of the first $p$ segments of length $k$ of $\omega_i$ (the condition $K(d,k)>2(d-1)k$ ensures there are enough such segments).
Then $(\omega_1',\omega_2')$ defines a $(d+1,k)$ circuit code of length $N+2p=2(q+1)s$, thus proving the claim.
\end{proof}
\begin{cor}
\label{cor1}
Let $k\ge 2$ be even and let $m$ and $n$ be integers such that $(k+2)^2\le m\le n$.
Then $K(m+n+2,k)\ge (1/k)K(m,k-1)K(n,k)$.
\end{cor}
\begin{proof}
Theorem 1 of \cite{Klee} implies that $K(m+n+2,k)\ge (1/k)K(m+1,k-1,k)K(n,k)$.  
From Lemma \ref{lem_divis}, $K(m+1,k-1,k)\ge K(m,k-1)$.
\end{proof}

\section{Proof of Theorem \ref{thm1}}
\begin{prop}
\label{lem_add}
For $k$ even, $\log_2K(d-3,k)$ is superadditive (in $d$) up to a constant.
\end{prop}
\begin{proof}
For $n\ge  m\ge (k+2)^2+3$, we have
\begin{align}
K(m-3,k-1)K(n-3,k)& \le K(m-2,k-1,k)K(n-3,k) & \text{ by Lemma \ref{lem_divis}} \nonumber \\
& \le k K(m+n-3,k) & \text{ by Corollary \ref{cor1}} \nonumber
\end{align}
That is, $a_m+b_n\le a_{m+n}$, with
$a_n:=\log_2(K(n-3,k)/k)$, $b_n:=\log_2(K(n-3,k-1)/k)$
nondecreasing sequences satisfying $a_n\le b_n \le n$ by construction.
It follows that $a_m+a_n\le a_{m+n}$, hence $\{a_n\}$ is superadditive and $\log_2 K(d-3,k)$ is superadditive up to a constant.
\end{proof}

\begin{cor}
\label{cor_nu}
The limit (\ref{nu_def}) exists and $\nu(k-1)=\nu(k)$.
\end{cor}
\begin{proof}

Let $a_n$ and $b_n$ be as in Proposition \ref{lem_add}.
It follows from Fekete's Lemma (\cite{Fekete}, \cite{Steele}) for superadditive functions
that $\lim_n\frac{a_n}{n}=\limsup_n\frac{a_n}{n}$, so $(\ref{nu_def})$ exists for even $k$.
Furthermore, $a_n\le b_n \le a_{2n}-a_n$ implies $\lim_n \frac{b_n}{n} = \lim_n \frac{a_n}{n}$, hence $\nu(k-1)=\nu(k)$. 
\end{proof}

\begin{prop}
\label{prop_pn}
The limit $(\ref{nu_def})$ equals $1$ on a subsequence $d_i\to+\infty$.
\end{prop}
\begin{proof}
We rely on the following construction of Preparata and Nievergelt \cite{PN}: for $i\ge 1$,
$d_i=2^i+2i+5k-4$, there exists a distinguished difference-preserving code $P$
(an open path in $I(d)$ satisfying the \emph{spread} $k$ condition $d_{I(d)}(x,y) \ge \min\{d_P(x,y),k\} \ \forall x,y\in P$) 
in dimension $\hat{d}_i=d_i-k$, of length $N\ge 2^{\hat{d}_i-(2+k)i-c}$, with $c$ a fixed constant.
We extend $P$ to a \emph{circuit code} as follows:
if $T(P)$ is the transition sequence of $P$, $(T(P),\hat{d}_i+1,\ldots,\hat{d}_i+k,T(P),\hat{d}_i+1,\ldots,\hat{d}_i+k)$ 
is the transition sequence of a $(d_i,k)$ circuit code of length $2(N+k)$.  
This implies  $K(d_i,k)\ge 2(N+k)$, hence
$\lim_{i\to\infty}\log_2K(d_i,k)/d_i=1$.  
\end{proof}

\begin{proof}[Proof of Theorem \ref{thm1}]
Corollary \ref{cor_nu} and Proposition \ref{prop_pn} yield $\nu(k)=1$, $\forall k\ge 1$. 
To obtain the rate of convergence we exploit two particular features of the sequence $d_i=2^i+2i+5k-1$:
1) any integer $d\ge 1$ can be written as ${d=p+\sum_{i=1}^H\epsilon_i d_i}$, 
with $0 \le p<d_1$ and $\epsilon_i\in\{0,1\}$
(by induction);
2) ${a_{d_i}=d_i + O(\log d_i)}$ with $a_n$ as in Proposition \ref{lem_add} (by Proposition \ref{prop_pn}).
If $k$ is even, Proposition \ref{lem_add} yields
$a_d\ge \sum_{i=1}^H \epsilon_i a_{d_i} + O(1) = \sum_{i=1}^H \epsilon_i (d_i +O(\log d_i)) + O(1)$.
This in turn equals $d+O(\sum_{i=1}^H \log d_i)+O(1)=d+O(\log^2d)$, implying $\log_2 K(d,k)$ is of this form as well. 
For $k$ odd, the claim follows from $K(d,k) \ge K(d,k+1)$.
\end{proof}

\section{Concluding Remarks}
In this note we presented a proof that a maximal $(d,k)$ circuit code has size $2^{d+O_k(\log^2d)}$. 
In \cite{PN} the authors construct difference-preserving codes of this same asymptotic order.  However, that construction is based on extending a binary BCH code in dimension $2^i-1$ having minimum distance $2k+1$ and length $2^{2^i-1-ik}$ to a $(d,k)$ difference-preserving code in dimension $d=2^i+O(i)$, and thus seems limited to a sparse subsequence of dimensions. 
Here we prove this result for circuit codes and the full sequence 
$d\to\infty$ by observing that available inequalities on maximizing circuit codes 
translate into the superadditivity (up to a constant) of the sequence $\log_2K(d-3,k)$, when $k$ is even.

\footnotesize
\bibliographystyle{plain}
\bibliography{CircuitCodesChar}

\begin{thebibliography}{1}

\bibitem{AK}
H.~L. Abbot and M.~Katchalski.
\newblock On the construction of snake in the box codes.
\newblock {\em Utilitas Mathematica}, 40:97--116, 1991.

\bibitem{Byrnes2019}
K.~M. Byrnes.
\newblock The maximum length of circuit codes with long bit runs and a new
  characterization theorem.
\newblock {\em Designs, Codes and Cryptography}, May 2019.

\bibitem{Fekete}
M.~Fekete.
\newblock {\"U}ber die verteilung der wurzeln bei gewissen algebraischen
  gleichungen mit ganzzahligen koeffizienten.
\newblock {\em Mathematische Zeitschrift}, 17(1):228--249, 1923.

\bibitem{Grunbaum}
B.~Gr{\"u}nbaum, G.C. Shephard, and V.~Klee.
\newblock {\em Convex Polytopes}.
\newblock Springer-Verlag, New York, New York, 2003.

\bibitem{Klee}
V.~Klee.
\newblock A method for constructing circuit codes.
\newblock {\em Journal of the ACM}, 14(3):520--528, 1967.

\bibitem{PN}
F.~P. Preparata and J.~Nievergelt.
\newblock Difference-preserving codes.
\newblock {\em IEEE Transactions on Information Theory}, 20(5):643--649, 1974.

\bibitem{Singleton}
R.~C. Singleton.
\newblock Generalized snake-in-the-box codes.
\newblock {\em IEEE Trans. Electronic Computers}, 15:596--602, 1966.

\bibitem{Steele}
J.M. Steele.
\newblock {\em Probability Theory and Combinatorial Optimization}.
\newblock CBMS-NSF Regional Conference Series in Applied Mathematics. Society
  for Industrial and Applied Mathematics, Philadelpha, 1997.

\end{thebibliography}
\normalsize

\end{document}